\documentclass[11pt, leqno, oneside]{amsart}

\usepackage{amsmath, amsfonts, amssymb, amsthm, mathtools}
\usepackage{enumerate}  
\usepackage{enumitem}
\usepackage{multicol}
\usepackage{stmaryrd}
\usepackage{fullpage}
\usepackage{setspace}
\usepackage{tikz}
\usepackage{tikz-cd}
\usepackage{geometry}
\usepackage{hyperref}
\usepackage{booktabs}

\hypersetup{
	colorlinks = true,
	citecolor = violet,
  	linkcolor  = cyan,
  	urlcolor   = violet
}

\numberwithin{equation}{section}

\newtheorem{thm}{Theorem}[section]
\newtheorem{cor}[thm]{Corollary}
\newtheorem{lem}[thm]{Lemma}
\newtheorem{prop}[thm]{Proposition}
\newtheorem{thmx}{Theorem}

\theoremstyle{definition}
\newtheorem{defi}[thm]{Definition}

\newtheorem{remark}[thm]{Remark}
\newtheorem{example}[thm]{Example}

\newcommand{\Z}{\mathbb{Z}}

\DeclareMathOperator{\rank}{rank}

\DeclareMathOperator{\cok}{cok}
\renewcommand{\phi}{\varphi}
\DeclareMathOperator{\Ima}{Im}

\DeclareMathOperator{\End}{End}

\DeclareMathOperator{\Hom}{Hom}

\DeclareMathOperator{\dep}{depth}

\DeclareMathOperator{\syz}{syz}

\DeclareMathOperator{\rad}{rad}

\DeclareMathOperator{\charac}{char}
\DeclareMathOperator{\ord}{ord}

\newcommand{\MatFac}[3]{\textup{MF}_{#1}^{#2}(#3)}
\newcommand{\MCM}{\textup{MCM}}

\newcommand{\mc}{\mathcal}
\newcommand{\mf}[1]{\mathfrak{#1}}

\title{Branched covers and matrix factorizations}
\author{Graham J. Leuschke and Tim Tribone}

\begin{document}
\maketitle
\begin{abstract}
    Let $(S,\mf n)$ be a regular local ring and $f$ a non-zero element of $\mf n^2$. A theorem due to Kn\"orrer states that there are finitely many isomorphism classes of maximal Cohen-Macaulay $R=S/(f)$-modules if and only if the same is true for the double branched cover of $R$, that is, the hypersurface ring defined by $f+z^2$ in $S\llbracket z \rrbracket$. We consider an analogue of this statement in the case of the hypersurface ring defined instead by $f+z^d$ for $d\ge 2$. In particular, we show that this hypersurface, which we refer to as the $d$-fold branched cover of $R$, has finite Cohen-Macaulay representation type if and only if, up to isomorphism, there are only finitely many indecomposable matrix factorizations of $f$ with $d$ factors. As a result, we give a complete list of polynomials $f$ with this property in characteristic zero. Furthermore, we show that reduced $d$-fold matrix factorizations of $f$ correspond to Ulrich modules over the $d$-fold branched cover of $R$.
\end{abstract}

\thispagestyle{empty} 

\section{Introduction}

Let $(S,\mf n, \mathbf{k})$ be a complete regular local ring and let $d \ge 2$ be an integer. Assume that $\mathbf{k}$ is algebraically closed and that the characteristic of $\mathbf{k}$ does not divide $d$. Fix a non-zero element $f \in \mf n^2$ and let $R=S/(f)$ be the hypersurface ring defined by $f$.

A finitely generated module $M$ over a local ring $A$ is called \textit{maximal Cohen-Macaulay} (MCM) if $\dep_A(M)$ $= \dim(A)$, the Krull dimension of $A$. We will consider MCM modules over $R$ and over the \textit{$d$-fold branched cover} of $R$, that is, the hypersurface ring $R^\sharp = S\llbracket z\rrbracket/(f+z^d)$. The main objective of this paper is to further understand the connection, established for $d=2$ by Kn\"orrer \cite{knorrer_cohen-macaulay_1987} and extended for $d> 2$ in \cite{tribone_matrix_2021}, between MCM $R^\sharp$-modules and matrix factorizations of $f$ with $d$ factors which we define below.
\begin{defi}
A \textit{matrix factorization of $f$ with $d$ factors} is a $d$-tuple of homomorphisms \[X=(\phi_1:F_2\to F_1,\phi_2: F_3\to F_2,\dots,\phi_d:F_1 \to F_d)\] between finitely generated free $S$-modules of the same rank such that $\phi_1\phi_2\cdots\phi_d = f\cdot 1_{F_1}$. We denote the category of matrix factorizations of $f$ with $d$ factors by $\MatFac{S}{d}{f}$. If $\rank_S(F_i) = n$ for all $i$, then we say $X$ is of size $n$.
\end{defi}

A local ring $A$ is said to have \textit{finite Cohen-Macaulay} (CM) \textit{type} if there are only finitely many isomorphism classes of indecomposable objects in the category $\MCM(A)$ of MCM $A$-modules. We adopt the following analogous terminology for the representation type of $\MatFac{S}{d}{f}$.
\begin{defi}\label{def:finite_dMF_type}
We say that $f$ has \textit{finite $d$-\textup{MF} type} if the category $\MatFac{S}{d}{f}$ has, up to isomorphism, only finitely many indecomposable objects.  
\end{defi}

In \cite{knorrer_cohen-macaulay_1987}, Kn\"orrer proved that $R=S/(f)$ has finite CM type if and only if $R^\sharp = S\llbracket z \rrbracket/(f+z^2)$ has finite CM type. The correspondence, given by Eisenbud \cite[Corollary 6.3]{eisenbud_homological_1980}, between matrix factorizations and MCM $R$-modules implies that the number of isomorphism classes of indecomposable objects in $\MCM(R)$ and $\MatFac{S}{2}{f}$ differ by only one. Since $R^\sharp$ is also a hypersurface ring, the same is true for $\MCM(R^\sharp)$ and $\MatFac{S\llbracket z \rrbracket}{2}{f+z^2}$. With this in mind, we can state the following version of Kn\"orrer's theorem.

\begin{thm}\cite[Corollary 2.8]{knorrer_cohen-macaulay_1987}\label{thm:knorrer_4_cats}
Let $f \in \mf n^2$ be non-zero, $R=S/(f)$, and $d=2$ so that $R^\sharp = S\llbracket z \rrbracket/(f+z^2)$ and $\charac \mathbf{k}\neq 2$. Then the following are equivalent:
\begin{enumerate}[label = (\roman*)]
    \item $R$ has finite CM type;
    \item $f$ has finite $2$-MF type;
    \item $R^\sharp$ has finite CM type;
    \item $f+z^2$ has finite $2$-MF type.
\end{enumerate}
\end{thm}

In Section \ref{section:rep_type} we investigate which of the analogous implications for $d$-fold factorizations hold when $d \ge 2$. Our main result in this direction is the following.

\begin{thmx}\label{thm:finite_dMF_type_2}
Let $d \ge 2$. Then $f$ has finite $d$-MF type if and only if the $d$-fold branched cover $R^\sharp = S\llbracket z \rrbracket/(f+z^d)$ has finite CM type.
\end{thmx}
As a consequence of Theorem \ref{thm:finite_dMF_type_2}, we give a complete list of polynomials which have only finitely many indecomposable $d$-fold matrix factorizations up to isomorphism.

\begin{thmx}\label{thm:B}
Let $\bf k$ be an algebraically closed field of characteristic zero and $S=\mathbf{k}\llbracket y,x_2,\dots,x_r\rrbracket$. Assume $0\neq f \in (y,x_2,\dots,x_r)^2$ and $d>2$. Then $f$ has finite $d$-\textup{MF} type if and only if, after a possible change of variables, $f$ and $d$ are one of the following:
\begin{enumerate}
    \item[$(A_1)$] \quad $y^2 + x_2^2 + \cdots + x_r^2$ for any $d>2$
    \item[$(A_2)$] \quad $y^3 + x_2^2 + \cdots + x_r^2$ for $d=3,4,5$
    \item[$(A_3)$] \quad $y^4 + x_2^2 + \cdots + x_r^2$ for $d=3$
    \item[$(A_4)$] \quad $y^5 + x_2^2 + \cdots + x_r^2$ for $d=3$.
\end{enumerate}
\end{thmx}

The main tool in Section \ref{section:rep_type} is a pair of functors $(-)^\flat : \MCM(R^\sharp) \to \MatFac{S}{d}{f}$ and $(-)^\sharp :\MatFac{S}{d}{f} \to \MCM(R^\sharp)$. The functors we define play a similar role as the ones in \cite[Chapter 8, \S 2]{leuschke_cohen-macaulay_2012} given by the same notation. Following the results in \cite[Chapter 8, \S3]{leuschke_cohen-macaulay_2012} and \cite[Proposition 2.7]{knorrer_cohen-macaulay_1987}, we investigate the decomposability of $(-)^\flat$ and $(-)^\sharp$ in Section \ref{section:decomp_sharp_flat}.

A matrix factorization $X= (\phi_1,\dots,\phi_d) \in \MatFac{S}{d}{f}$ is called \textit{reduced} if the map $\phi_k$ is minimal for each $k \in \Z_d$. Equivalently, $X$ is reduced if, after choosing bases, all entries of $\phi_k$ lie in the maximal ideal of $S$ for all $k \in \Z_d$. In the original setting given by Eisenbud ($d=2$), matrix factorizations which are reduced are the only interesting ones. In particular, the only indecomposable non-reduced matrix factorizations with $2$ factors are $(1,f)$ and $(f,1)$. For $d >2$, there are often non-reduced $d$-fold matrix factorizations of $f$ which are also indecomposable and therefore play a role in Theorem \ref{thm:finite_dMF_type_2} and Theorem \ref{thm:B}. 

In Section \ref{section:reduced_dMFs}, we focus only on reduced matrix factorizations. We show that, as long as $d$ is not too large, reduced matrix factorizations of $f$ with $d$ factors correspond to \textit{Ulrich modules} over $R^\sharp$, that is, modules which are MCM and maximally generated in the sense of \cite{brennan_maximally_1987}. We let $\text{ord}(f)$ denote the maximal integer $e$ such that $f \in \mf n^e$.

\begin{thmx}\label{thm:C}
Assume that $d \leq \textup{ord}(f)$ and that $f+z^d$ is irreducible. Then $f$ has finite reduced $d$-MF type if and only if there are, up to isomorphism, only finitely many indecomposable Ulrich $R^\sharp$-modules.
\end{thmx}

\subsection*{Notation} All indices are taken modulo $d$ unless otherwise specified.

\subsection*{Acknowledgements} The first-named author was supported by a CUSE grant from Syracuse University.

\section{Finite matrix factorization type}\label{section:rep_type}

The key lemma in Kn\"orrer's proof of Theorem~\ref{thm:knorrer_4_cats} states that for  MCM modules $M$ over $R$ and $N$ over the double branched cover $R^\# = S\llbracket z \rrbracket/(f+z^2)$, we have isomorphisms
\[
\syz^{R^\#}_1(M)/z \syz^{R^\#}_1(M) \cong M \oplus \syz^R_1(M)
\]
and 
\[
\syz^{R^\#}_1(N/zN) \cong N \oplus \syz^{R^\#}_1(N)\,.
\]
In this section we extend Kn\"orrer's lemma to arbitrary $d$-fold branched covers and use it to prove Theorem~\ref{thm:finite_dMF_type_2}.

As in the Introduction, we fix $d\ge 2$ and let $(S,\mf n, \mathbf{k})$ be a complete regular local ring with $\mathbf{k}$ algebraically closed of characteristic not dividing $d$. We also fix a non-zero element $f\in \mf n^2$ and let $R=S/(f)$. We recall a few facts about matrix factorizations of $f$ with two or more factors.

\begin{lem}[\cite{tribone_matrix_2021}] Let $X=(\phi_1,\phi_2,\dots,\phi_d) \in \MatFac{S}{d}{f}$ and $i \in \Z_d$.
\begin{enumerate}[label = (\roman*)]
    \item The $d$-tuple $(\phi_i,\phi_{i+1},\dots,\phi_d,\phi_1,\dots,\phi_{i-1})$ also forms a matrix factorization of $f$ with $d$ factors, that is, $\phi_i\phi_{i+1}\cdots\phi_d\phi_1\cdots\phi_{i-1}=f\cdot 1_{F_i}$.
    \item The $R$-module $\cok\phi_i$ is maximal Cohen-Macaulay.
    \item The Krull-Remak-Schmidt Theorem (KRS) holds in the additive category\footnote{Morphisms and direct sums in $\MatFac{S}{d}{f}$ are defined component-wise (see \cite[1.4]{tribone_matrix_2021}). The zero object is the matrix factorization with $F_i = 0$ and $\phi_i = 0$ for all $i \in \Z_d$.} $\MatFac{S}{d}{f}$ (since $S$ is complete). That is, any object $X \in \MatFac{S}{d}{f}$ decomposes into a finite direct sum of indecomposable objects such that the decomposition is unique up to isomorphism and permutation of the summands.
\end{enumerate}

\end{lem}

The $d$-fold branched cover of $R$ is the hypersurface ring $R^\sharp = S\llbracket z \rrbracket/(f+z^d)$. As an $S$-module, $R^\sharp$ is finitely generated and free with basis given by $\{1,z,z^2,\dots,z^{d-1}\}$. Consequently, a finitely generated $R^\sharp$-module $N$ is MCM over $R^\sharp$ if and only if it is free over $S$ \cite[Proposition 1.9]{yoshino_cohen-macaulay_1990}. Furthermore, multiplication by $z$ on $N$ defines an $S$-linear map $\phi:N\to N$ which satisfies $\phi^d = -f\cdot 1_N$. Conversely, given a free $S$-module $F$ and a homomorphism $\phi:F \to F$ satisfying $\phi^d = -f\cdot 1_F$, the pair $(F,\phi)$ defines an MCM $R^\sharp$-module whose $z$-action is given by the map $\phi$. We will use these perspectives interchangeably throughout.

Since $S$ is complete and $\charac \mathbf{k} \nmid d$, \cite[A.31]{leuschke_cohen-macaulay_2012} implies that the polynomials $x^d+1$ and $x^d-1$ in $S[x]$ each have $d$ distinct roots in $S$. We let $\omega \in S$ be a \textit{primitive} $d^{\text{th}}$ root of $1$ in the sense that $\omega^d = 1$ and $\omega^s \neq 1$ for any $1\leq s < d$. Let $\mu \in S$ be any $d^{\text{th}}$ root of $-1$. We start with a pair of functors between the categories $\MCM(R^\sharp)$ and $\MatFac{S}{d}{f}$.

\begin{defi}\label{def:flat_sharp}\
\begin{enumerate}[label = (\roman*)]
    \item Let $N \in \MCM(R^\sharp)$ and set $\phi: N \to N$ to be the $S$-linear homomorphism representing multiplication by $z$ on $N$. Since $(\mu\phi)^d = f \cdot 1_N$, we define
\[N^\flat = (\mu\phi,\mu\phi,\dots,\mu\phi) \in \MatFac{S}{d}{f}.\] Given a homomorphism $g: N \to N'$ of $R^\sharp$-modules, we define $g^\flat = (g,g,\dots,g)$ which is a morphism of matrix factorizations in $\Hom_{\MatFac{S}{d}{f}}(N^\flat,(N')^\flat)$. These operations define a functor $( - )^\flat : \MCM(R^\sharp) \to \MatFac{S}{d}{f}$.
\item Let $X=(\phi_1:F_2 \to F_1,\phi_2:F_3\to F_2,\dots,\phi_d:F_1\to F_d) \in \MatFac{S}{d}{f}$. Define an MCM $R^\sharp$-module by setting $X^\sharp = \bigoplus_{k=0}^{d-1}F_{d-k}$ as an $S$-module with $z$-action given by:
\[z\cdot(x_d,x_{d-1},\dots,x_2,x_1) \coloneqq (\mu^{-1}\phi_d(x_1),\mu^{-1}\phi_{d-1}(x_d),\dots,\mu^{-1}\phi_1(x_2))\] for all $x_i \in F_i, i \in \Z_d$. Given a morphism $\alpha=(\alpha_1,\alpha_2,\dots,\alpha_d): X \to X'$, define an $R^\sharp$-homomorphism $\alpha^\sharp = \bigoplus_{k=0}^{d-1}\alpha_{d-k} \in \Hom_{R^\sharp}(X^\sharp,(X')^\sharp)$. This gives a functor $(-)^\sharp:\MatFac{S}{d}{f} \to \MCM(R^\sharp)$.
\end{enumerate}
\end{defi}

\begin{remark}\label{remark:mu}
The role of $\mu$ in the definition of $N^\flat$ is to obtain a $d$-fold factorization of $f$ (instead of $-f$) and to do so in a symmetric way. It is important to note that the isomorphism class of $N^\flat \in \MatFac{S}{d}{f}$ is independent of choice of $\mu$. To see this, observe that if $\mu'$ is another root of $x^d + 1$, then $\mu' = \omega^j \mu$ for some $j \in \Z_d$. Using this fact, one can construct an isomorphism $(\mu\phi_N,\mu\phi_N,\dots,\mu\phi_N) \cong (\mu'\phi_N,\mu'\phi_N,\dots,\mu'\phi_N)$ in $\MatFac{S}{d}{f}$. Similarly, $\mu^{-1}$ in the definition of $X^\sharp$ ensures that we obtain a module over $R^\sharp$. The isomorphism class of $X^\sharp$ in $\MCM(R^\sharp)$ is also independent of the choice of $\mu$.
\end{remark}

Consider the automorphism $\sigma:R^\sharp \to R^\sharp$ which fixes $S$ pointwise and maps $z$ to $\omega z$. This automorphism acts on the category of MCM $R^\sharp$-modules in the following sense: For each $N \in \MCM(R^\sharp)$, let $(\sigma^k)^*N$ denote the MCM $R^\sharp$-module obtained by restricting scalars along $\sigma^k$. Since $\sigma^d = 1_{R^{\sharp}}$, the mapping $N \mapsto \sigma^*N$ forms an autoequivalence of the category $\MCM(R^\sharp)$.

The functor $T: \MatFac{S}{d}{f} \to \MatFac{S}{d}{f}$ given by $T(\phi_1,\phi_2,\dots,\phi_d) = (\phi_2,\phi_3,\dots,\phi_d,\phi_1)$, which we refer to as the \textit{shift functor}, also gives an autoequivalence, this time of $\MatFac{S}{d}{f}$.

\begin{prop}\label{thm:flat_sharp_and_sharp_flat}
Let $N$ be an MCM $R^\sharp$-module and $X \in \MatFac{S}{d}{f}$. Then
\[X^{\sharp\flat} \cong \bigoplus_{k \in \Z_d}T^k(X) \quad\text{ and }\quad N^{\flat\sharp} \cong \bigoplus_{k \in \Z_d}(\sigma^k)^*N.\]
\end{prop}

\begin{proof}
Let $X = (\phi_1:F_2\to F_1,\dots,\phi_d:F_1\to F_d) \in \MatFac{S}{d}{f}$. Since $X^\sharp = F_d\oplus F_{d-1}\oplus \cdots \oplus F_1$ as an $S$-module, multiplication by $z$ on $X^\sharp$ is given by $\mu^{-1}\phi$ where
\[\phi = \begin{pmatrix}
0 &0 &\cdots &0 &\phi_d\\
\phi_{d-1} &0 &\cdots &0 &0\\
0 &\phi_{d-2} &\ddots &\vdots &\vdots\\
\vdots &\vdots &\ddots &\ddots& \vdots\\
0 &0 &\cdots          &\phi_1 &0
\end{pmatrix}.\] Therefore, $X^{\sharp\flat} = (\phi,\phi,\dots,\phi) \in \MatFac{S}{d}{f}$. One can perform row and column operations to see that that $(\phi,\phi,\dots,\phi) \cong \bigoplus_{k\in \Z_d}T^k(X)$.

Notice that the first half of the proof is valid in any characteristic as long as there exists an element $\mu \in S$ satisfying $\mu^d = -1$. For instance, if $d$ is odd, then $\mu = -1$ is a valid choice. However, the second half of the proof explicitly makes use of the fact that $\charac \bf k$ does not divide $d$.

In order to show the second isomorphism, let $N \in \MCM(R^\sharp)$ and set $\phi: N \to N$ to be the $S$-linear map representing multiplication by $z$ on $N$. Then
\(N^\flat = (\mu\phi:N \to N,\dots,\mu\phi: N \to N) \in \MatFac{S}{d}{f}\). Thus, as an $S$-module,
\(N^{\flat\sharp} = N \oplus N \oplus \cdots \oplus N,\) the direct sum of $d$ copies of the free $S$-module $N$. The $z$-action on $N^{\flat\sharp}$ is given by 
\[\begin{split}z \cdot (n_d,n_{d-1},\dots,n_1) &= (\mu^{-1}\mu\phi(n_1),\mu^{-1}\mu\phi(n_d),\dots,\mu^{-1}\mu\phi(n_2))\\
&= (zn_1,zn_d,\dots,zn_2),
\end{split}\] for any $n_i \in N , i \in \Z_d$.
Let $k \in \Z_d$ and define a map $g_k: N^{\flat\sharp} \to (\sigma^k)^*N$ by mapping
\[n=(n_d,n_{d-1},\dots,n_1) \mapsto \frac{1}{d}\sum_{j=0}^{d-1}\omega^{jk}n_{d-j}\] for any $n \in N^{\flat\sharp}$.
Note that for $m \in (\sigma^k)^*N$, $z\cdot m = \omega^k z m$ by definition. Therefore, for $n=(n_d,\dots,n_1) \in N^{\flat\sharp}$, we have that
\[\begin{split}
    z\cdot g_k(n) &= \frac{1}{d}\sum_{j=0}^{d-1}z\cdot(\omega^{jk}n_{d-j})\\
                  &= \frac{1}{d}\sum_{j=0}^{d-1}\omega^{(j+1)k}zn_{d-j}\\
                  &= \frac{1}{d}\left(\omega^k z n_{d} + \omega^{2k}zn_{d-1} + \cdots + \omega^{(d-1)k}zn_{2} + zn_1\right)\\
                  &= g_k(z \cdot n).
\end{split}\]
In other words, $g_k$ is an $R^\sharp$-homomorphism. Putting these maps together we have an $R^\sharp$-homomorphism 
\[g = \begin{pmatrix}g_0\\ g_1 \\ \vdots \\ g_{d-1} \end{pmatrix}: N^{\flat\sharp} \to \bigoplus_{k=0}^{d-1}(\sigma^k)^*N.\] In the other direction, we have $R^\sharp$-homomorphisms $s_k: (\sigma^k)^*N \to N^{\flat\sharp}$, $k \in \Z_d$, given by
\[s_k(m) = (m, \omega^{-k}m,\omega^{-2k}m,\dots,\omega^{-(d-1)k}m)\] for any $m \in (\sigma^k)^*N$. For each $k \in \Z_d$ and $m \in (\sigma^k)^*N$, we have that $g_ks_k(m) = m$. On the other hand, if $i\neq \ell \in \Z_d$, then
\[\begin{split}
    g_is_\ell(m) &= g_i(m,\omega^{-\ell} m,\omega^{-2\ell}m,\dots,\omega^\ell m)\\
                 &= \frac{1}{d}\sum_{j=0}^{d-1}\omega^{j(i-\ell)}m\\
                 &=0.
\end{split}\] Therefore, setting $s = \begin{pmatrix}s_0 & s_1 &\cdots &s_{d-1}\end{pmatrix}$, we have that $gs$ is the identity on $\bigoplus_{k=0}^{d-1}(\sigma^k)^*N$ and so $g$ is a split surjection. However, since both the target and source of $g$ have the same rank as free $S$-modules, we can conclude that $g$ is an isomorphism of $R^\sharp$-modules.
\end{proof}

\begin{cor}\label{thm:cor_knorrer}
\
\begin{enumerate}[label = (\roman*)]
    \item For each $X \in \MatFac{S}{d}{f}$, there exists $N \in \MCM(R^\sharp)$ such that $X$ is isomorphic to a summand of $N^\flat$.
    \item For each $N \in \MCM(R^\sharp)$, there exists $X \in \MatFac{S}{d}{f}$ such that $N$ is isomorphic to a summand of $X^\sharp$.\qed
\end{enumerate}
\end{cor}

We can now prove Theorem \ref{thm:finite_dMF_type_2}. The proof is lifted directly from the $d=2$ case. Once again, the characteristic assumption on $\charac \bf k$ is only needed in half of the proof as long as there exists $\mu \in S$ satisfying $\mu^d = -1$.

\begin{proof}[Proof of Theorem \ref{thm:finite_dMF_type_2}]
Let $X_1,X_2,\dots,X_t$ be a representative list of the isomorphism classes of indecomposable $d$-fold matrix factorizations of $f$ and let $N \in \MCM(R^\sharp)$ be indecomposable. Since $N^\flat \in \MatFac{S}{d}{f}$, there exist non-negative integers $s_1,s_2,\dots,s_t$ such that $N^\flat \cong X_1^{s_1}\oplus X_2^{s_2}\oplus \cdots\oplus X_t^{s_t}$. By Proposition \ref{thm:flat_sharp_and_sharp_flat}, $N$ is isomorphic to a direct summand of
\[N^{\flat\sharp} \cong (X_1^\sharp)^{s_1}\oplus (X_2^\sharp)^{s_2} \oplus \cdots \oplus (X_t^\sharp)^{s_t}.\] Since $N$ is indecomposable, KRS in $\MCM(R^\sharp)$ implies that $N$ is isomorphic to a summand of $X_i^\sharp$ for some $1\leq i \leq t$. Hence, every indecomposable MCM $R^\sharp$-module is isomorphic to one appearing in the finite list consisting of all summands of all $X_j^\sharp$, $1\leq j \leq t$. The converse follows similarly from Proposition \ref{thm:flat_sharp_and_sharp_flat} and KRS in $\MatFac{S}{d}{f}$.
\end{proof}

As above, $\sigma: R^\sharp \to R^\sharp$ is the automorphism of $R^\sharp$ which pointwise fixes $S$ and and maps $z$ to $\omega z$. In \cite[Section 4]{tribone_matrix_2021}, it was shown the the category $\MatFac{S}{d}{f}$ is equivalent to the category of finitely generated modules over the skew group algebra $R^\sharp[\sigma]$ which are MCM as $R^\sharp$-modules (equivalently, free as $S$-modules). We denote this category by $\MCM_{\sigma}(R^\sharp)$. The equivalence is given by a pair of inverse functors $\mc A:\MCM_{\sigma}(R^\sharp) \to \MatFac{S}{d}{f}$ and $\mc B: \MatFac{S}{d}{f} \to \MCM_{\sigma}(R^\sharp)$ defined as follows: For $X = (\phi_1:F_2 \to F_1,\dots,\phi_d: F_1\to F_d) \in \MatFac{S}{d}{f}$, $\mc B(X) \coloneqq X^\sharp$ as an $R^\sharp$-module with the action of $\sigma$ given by
\[\sigma\cdot (x_d,x_{d-1},\dots,x_1) = (x_d,\omega x_{d-1},\dots,\omega^{d-1}x_1)\] for all $x_i \in F_i, i \in \Z_d$. For any $N \in \MCM_{\sigma}(R^\sharp)$, $N$ decomposes as an $S$-module into $N \cong \bigoplus_{i \in \Z_d}N^{\omega^i}$ \cite[Lemma 4.2]{tribone_matrix_2021} where $N^{\omega^i}$ is the $\sigma$-invariant subspace 
\(N^{\omega^i} = \left\{ x \in N : \sigma(x) = \omega^ix\right\}.\) For each $i \in \Z_d$, $zN^{\omega^i} \subseteq N^{\omega^{i+1}}$ and therefore multiplication by $z$ defines an $S$-linear map $N^{\omega^i}\to N^{\omega^{i+1}}$. Since $(\mu z)^d = f$, the composition
\[\begin{tikzcd}N^{\omega^{d-1}} \rar{\mu z} &N^{1} \rar{\mu z} &N^{\omega} \rar{\mu z} &\cdots \rar{\mu z} &N^{\omega^{d-2}} \rar{\mu z} &N^{\omega^{d-1}}\end{tikzcd}\] defines a matrix factorization which we denote as $\mc A(N) \in \MatFac{S}{d}{f}$.

To finish this section, we make note of the connection between the functors $\mc A$ and $\mc B$ and the functors $(-)^\sharp$ and $(-)^\flat$.

\begin{lem}\label{thm:A_B_flat_sharp}
Let $H: \MCM_{\sigma}(R^\sharp) \to \MCM(R^\sharp)$ be the functor which forgets the action of $\sigma$ and $G:\MCM(R^\sharp) \to \MCM_{\sigma}(R^\sharp)$ be given by $G(N) = R^\sharp[\sigma] \otimes_{R^\sharp} N$ for any $N \in \MCM(R^\sharp)$.
\begin{enumerate}[label = (\roman*)]
    \item \label{thm:A_B_flat_sharp_1} For any $X \in \MatFac{S}{d}{f}$, $X^\sharp = H\circ\mc B(X)$.
    \item \label{thm:A_B_flat_sharp_2} For any $N \in \MCM(R^\sharp)$, $N^\flat \cong \mc A\circ G(N)$.
\end{enumerate}
\end{lem}

\begin{proof}
The first statement follows directly from the definition of $(-)^\sharp$ and $\mc B$. For the second, consider the idempotents
\[e_k = \frac{1}{d} \sum_{j \in \Z_d}\omega^{-jk}\sigma^j \in R^\sharp[\sigma], \quad k \in \Z_d.\] These idempotents have three important properties:
\begin{enumerate}[label = (\alph*)]
    \item \label{equation:e_k_properties_a} $R^\sharp[\sigma] = \bigoplus_{k \in \Z_d}e_kR^\sharp[\sigma]$ as right $R^\sharp[\sigma]$-modules,
    \item \label{equation:e_k_properties_b} $\sigma e_k = e_k \sigma = \omega^k e_k$, $k \in \Z_d$, and
    \item \label{equation:e_k_properties_c} $ze_k = e_{k-1}z$, $k \in \Z_d$.
\end{enumerate}

From \ref{equation:e_k_properties_b}, we have that $e_kR^\sharp[\sigma] = e_kR^\sharp$ where $e_kR^\sharp$ denotes the multiples of $e_k$ by $R^\sharp \cdot 1 \subset R^\sharp[\sigma]$ on the right. Hence, as an $R^\sharp$-module, $e_kR^\sharp[\sigma] = e_kR^\sharp$ is free of rank $1$. Thus, for any $N \in \MCM(R^\sharp)$, \ref{equation:e_k_properties_a} implies that $G(N) \cong \bigoplus_{k \in \Z_d} (e_kR^\sharp \otimes_{R^\sharp}N)$. It then follows from \ref{equation:e_k_properties_b} and \ref{equation:e_k_properties_c} that we have an isomorphism of $R^\sharp[\sigma]$-modules $G(N) \cong \mc B(N^\flat)$. Hence, $\mc A\circ G(N) \cong \mc A \circ \mc B(N^\flat) \cong N^\flat$ by \cite[Theorem 4.4]{tribone_matrix_2021}.
\end{proof}

\begin{remark}
In the case of an Artin algebra $\Lambda$, the relationship between $\Lambda$ and the skew group algebra $\Lambda[G]$ for a finite group $G$ was studied by Reiten and Riedtmann in \cite{reiten_skew_1985}. They show that many representation theoretic properties hold simultaneously for $\Lambda$ and $\Lambda[G]$. In particular, $\Lambda$ has finite representation type if and only if the same is true of $\Lambda[G]$. The equivalence of categories $\MatFac{S}{d}{f} \approx \MCM_{\sigma}(R^\sharp)$ \cite[Theorem 4.4]{tribone_matrix_2021} and Theorem \ref{thm:finite_dMF_type_2} give an analogous relationship between $R^\sharp$ and the skew group algebra $R^\sharp[\sigma]$.
\end{remark}

\section{Hypersurfaces of finite matrix factorization type}

Let $(A,\mf m)$ be a regular local ring and $g \in \mf m^2$ be non-zero. Then the hypersurface ring $A/(g)$ is called a \textit{simple hypersurface singularity} if there are only finitely many proper ideals $I \subset A$ such that $g \in I^2$. In the case that $A$ is a power series ring over an algebraically closed field of characteristic zero, the pair of papers \cite{buchweitz_cohen-macaulay_1987} and \cite{knorrer_cohen-macaulay_1987} prove the following theorem.

\begin{thm}[\cite{buchweitz_cohen-macaulay_1987},\cite{knorrer_cohen-macaulay_1987}]\label{thm:BGS_Knorrer}
Let $\mathbf{k}$ be an algebraically closed field of characteristic zero and let $R= \mathbf{k}\llbracket x_1,x_2,\dots,x_r\rrbracket/(g)$, where $g \in (x_1,x_2,\dots,x_r)^2$ is non-zero. Then $R$ has finite CM type if and only if $R$ is a simple hypersurface singularity.
\end{thm}

Essential to their conclusion is the classification of simple hypersurface singularities, due to Arnol'd \cite{arnold_normal_1973}, which gives explicit normal forms for all polynomials defining such a singularity. These are often referred to as the ADE singularities. The culmination of these results is a complete list of polynomials which define hypersurface rings of finite CM type in all dimensions (see \cite[Theorem 8.8]{yoshino_cohen-macaulay_1990} or \cite[Theorem 9.8]{leuschke_cohen-macaulay_2012}). Equivalently, the polynomials in this list are precisely the ones with only finitely many indecomposable $2$-fold matrix factorizations up to isomorphism.

Using Theorem \ref{thm:finite_dMF_type_2} and the classification described above, we are able to compile a list of all $f$ with finite $d$-MF type for $d>2$.

\begin{thm}\label{thm:finite_dMF_polys}
Let $\bf k$ be an algebraically closed field of characteristic zero and $S=\mathbf{k}\llbracket y,x_2,\dots,x_r\rrbracket$. Assume $0\neq f \in (y,x_2,\dots,x_r)^2$ and $d>2$. Then $f$ has finite $d$-\textup{MF} type if and only if, after a possible change of variables, $f$ and $d$ are one of the following:
\begin{enumerate}
    \item[$(A_1)$:] \quad $y^2 + x_2^2 + \cdots + x_r^2$ for any $d>2$
    \item[$(A_2)$:] \quad $y^3 + x_2^2 + \cdots + x_r^2$ for $d=3,4,5$
    \item[$(A_3)$:] \quad $y^4 + x_2^2 + \cdots + x_r^2$ for $d=3$
    \item[$(A_4)$:] \quad $y^5 + x_2^2 + \cdots + x_r^2$ for $d=3$.
\end{enumerate}
\end{thm}

\begin{proof}
Let $f$ and $d$ be a pair in the list given. Then $f+z^d \in S\llbracket z \rrbracket$ defines a simple hypersurface singularity and therefore $R^\sharp$ has finite CM type by Theorem \ref{thm:BGS_Knorrer}. By Theorem \ref{thm:finite_dMF_type_2}, $f$ has finite $d$-MF type. 

Conversely, let $0 \neq f \in (y,x_2,\dots,x_r)^2$, $d > 2$, and assume $f$ has finite $d$-MF type. Then $R^\sharp = S\llbracket z \rrbracket/(f+z^d)$ has dimension $r$ and is of finite CM type by Theorem \ref{thm:finite_dMF_type_2}. We consider two cases.

First, assume $\dim R^\sharp = 1$, that is, assume $S=\mathbf{k}\llbracket y \rrbracket$. Then $f = uy^k$ for some unit $u \in S$ and $k \ge 2$. Since $S$ is complete and $\charac \mathbf{k}=0$, there exists a $k$th root $v$ of $u^{-1}$ in $S$ \cite[A.31]{leuschke_cohen-macaulay_2012}. Therefore, after replacing $y$ with $vy$, we may assume that $f = y^k$. Since $\dim R^\sharp = 1$, \cite[8.2.1]{yoshino_cohen-macaulay_1990} implies that $\ord(y^k+z^d) \leq 3$. Hence, either $k \leq 3$ or $d \leq 3$. If $k=2$, there are no restrictions on $d$ since $y^2+z^d$ defines a simple ($A_{d-1}$) singularity for all $d > 2$. If $k=3$, then the fact that $y^3 + z^d$ defines a $1$-dimensional simple hypersurface singularity implies that $d=3,4,$ or $5$. Similarly, if $d=3$, then $k=2,3,4,$ or $5$.

Next, assume $\dim R^\sharp \ge 2$. In this case, \cite[8.2.2]{yoshino_cohen-macaulay_1990} implies that $\ord(f+z^d) \leq 2$. Since $d >2$ and $f \in (y,x_2,\dots,x_r)^2$, we have that $\ord(f)= 2$. By the Weierstrass Preparation Theorem \cite[Corollary 9.6]{leuschke_cohen-macaulay_2012}, there exists a unit $u \in S$ and $g \in \mathbf{k}\llbracket y,x_2,\dots,x_{r-1}\rrbracket$ such that $f = (g + x_r^2)u$.  As above, we may neglect the unit and assume that $f = g + x_{r}^2$ for some $g \in \mathbf{k}\llbracket y , x_2,\dots,x_{r-1}\rrbracket$.

Since $f+z^d = g + x_r^2 + z^d$ has finite $2$-MF type, Kn\"orrer's theorem (Theorem \ref{thm:knorrer_4_cats}) implies that $g + z^d$ has finite $2$-MF type as well. Thus, $g$ has finite $d$-MF type by Theorem \ref{thm:finite_dMF_type_2}. We repeat this argument until $f = g' + x_2^2 + \dots + x_r^2$ for some $g' \in \mathbf{k}\llbracket y \rrbracket$ with finite $d$-MF type. Finally, we apply the first case to $g'$ to finish the proof.
\end{proof}

\begin{cor}
Let $\bf k$ be an algebraically closed field of characteristic zero, $S=\mathbf{k}\llbracket y,x_2,\dots,x_r\rrbracket$, and $f \in (y,x_2,\dots,x_r)^2$ be non-zero. If $f$ has finite $d$-MF type for some $d\ge 2$, then $R=S/(f)$ is an isolated singularity, that is, $R_{\mf p}$ is a regular local ring for all non-maximal prime ideals $\mf p$.
\end{cor}

\begin{proof}
The polynomials listed in Theorem \ref{thm:finite_dMF_polys} are a subset of the ones in \cite[Theorem 8.8]{yoshino_cohen-macaulay_1990} (or \cite[Theorem 9.8]{leuschke_cohen-macaulay_2012}), all of which define isolated singularities. 
\end{proof}

\setlength{\abovedisplayskip}{0pt}
\setlength{\belowdisplayskip}{0pt}
\setlength{\abovedisplayshortskip}{0pt}
\setlength{\belowdisplayshortskip}{0pt}

Suppose we have a pair $f$ and $d$ from the list in Theorem \ref{thm:finite_dMF_polys} such that $R^\sharp$ has dimension $1$. Then \cite[Chapter 9]{yoshino_cohen-macaulay_1990} gives matrix factorizations for every indecomposable MCM $R^\sharp$-module. By computing multiplication by $z$ on each of these $R^\sharp$-modules, we can compile a representative list of all isomorphism classes of indecomposable $d$-fold factorizations of $f$. We give one such computation in the following example.

\begin{example}\label{example:E_6}
Let $\mathbf k$ be algebraically closed of characteristic zero. Let $S=\mathbf{k}\llbracket y \rrbracket$, $f=y^4 \in S$, and $R = S/(f)$. The hypersurface ring $R^\sharp = \mathbf{k}\llbracket x,y\rrbracket/(y^4+x^3)$ is a simple curve singularity of type $E_6$ and has finite CM type. Here we are viewing $R^\sharp$ as the $3$-fold branched cover of $R$. By Theorem \ref{thm:finite_dMF_type_2}, the category $\MatFac{S}{3}{y^4}$ has only finitely many non-isomorphic indecomposable objects. We give a complete list below.

A complete list of non-isomorphic indecomposable MCM $R^\sharp$-modules is given in \cite[9.13]{yoshino_cohen-macaulay_1990}. By Corollary \ref{thm:cor_knorrer}, we may compute multiplication by $x$ on each of these modules to obtain a representative from each isomorphism class of indecomposable matrix factorizations of $y^4$ with $3$ factors. By Remark \ref{remark:mu}, we may choose $\mu = -1$. 

Following the notation of \cite[9.13]{yoshino_cohen-macaulay_1990}, we let $\phi_1 = \begin{pmatrix}
x &y\\
y^3 &-x^2
\end{pmatrix}$ and $M_1 = \cok \phi_1$. Let $e_1$ and $e_2$ in $M_1$ denote the images of the standard basis on $S\llbracket x \rrbracket^2$. Then $e_1$ and $e_2$ satisfy $xe_1 = - y^3e_2$ and $x^2e_2 = ye_1$. As an $S$-module, $M_1$ is free with basis $\{e_1,e_2,xe_2\}$. Multiplication by $x$ on $M_1$ is therefore given by
\[ \phi = \begin{pmatrix}
    0 & 0 & y\\ -y^3 & 0 & 0 \\ 0 & 1 & 0
    \end{pmatrix}. \]
    Hence, $M_1^\flat = (-\phi,-\phi,-\phi) \in \MatFac{S}{3}{y^4}$. Furthermore, we have a commutative diagram
    
    \[\begin{tikzcd}[ampersand replacement = \&, column sep = 5em,row sep = 3em]
    S^3 \dar[swap]{\left(\begin{smallmatrix} 0 & 1 & 0 \\ -1 & 0 & 0 \\ 0 & 0 & -1\\\end{smallmatrix}\right)} \rar{-\phi} \&S^3 \dar[swap]{\left(\begin{smallmatrix} 0 & 0 & -1 \\ 0 & -1 & 0 \\ 1 & 0 & 0\\\end{smallmatrix}\right)} \rar{-\phi} \&S^3 \dar{\left(\begin{smallmatrix} 1 & 0 & 0 \\ 0 & 0 & 1\\ 0 & 1 & 0\\\end{smallmatrix}\right)} \rar{-\phi} \&S^3 \dar{\left(\begin{smallmatrix} 0 & 1 & 0 \\ -1 & 0 & 0 \\ 0 & 0 & -1 \\\end{smallmatrix}\right)}\\
    S^3 \rar[swap]{\left(\begin{smallmatrix}1 & 0 & 0 \\ 0 & y^3 & 0\\ 0 & 0 & y \end{smallmatrix}\right)} \&S^3 \rar[swap]{\left(\begin{smallmatrix}y & 0 & 0 \\ 0 & 1 & 0\\ 0 & 0 & y^3 \end{smallmatrix}\right)} \&S^3 \rar[swap]{\left(\begin{smallmatrix}y^3 & 0 & 0 \\ 0 & y & 0\\ 0 & 0 & 1 \end{smallmatrix}\right)} \&S^3.
    \end{tikzcd}\] 
    Thus, $M_1^\flat$ is isomorphic to the direct sum of the indecomposable factorization $X_{\phi_1}\coloneqq (y^3,y,1)$ and its corresponding shifts, that is, $M_1^\flat \cong \bigoplus_{i \in \Z_3}T^i(y^3,y,1)$.
    
    Similarly, multiplication by $x$ can be computed for each of the indecomposable MCM $R^\sharp$ modules listed in \cite[9.13]{yoshino_cohen-macaulay_1990}. From this computation, we obtain a list of $3$-fold matrix factorizations of $y^4$ given in the table below.

\begin{figure}
    \centering
    {\renewcommand{\arraystretch}{1.4}
    \begin{tabular}{lc}
    \toprule
    $X \in \MatFac{S}{3}{y^4}$ & $N^\flat$ for $N \in \MCM(R^\sharp)$\\
    \midrule
    $\mc P_1 = (y^4,1,1)$ & $(R^\sharp)^\flat \cong \bigoplus_{i \in \Z_3}T^i(\mc P_1)$\\
    $X_{\phi_1}= (y^3,y,1)$  & $M_1^\flat \cong \bigoplus_{i \in \Z_3}T^i(X_{\phi_1})$\\
    $X_{\psi_1}= (y^3,1,y)$  & $N_1^\flat \cong \bigoplus_{i \in \Z_3}T^i(X_{\psi_1})$\\
    $X_{\phi_2} = (y^2,y^2,1)$ & $M_2^\flat \cong N_2^\flat \cong \bigoplus_{i \in \Z_3}T^i(X_{\phi_2})$\\
    $X_{\beta} = (y^2,y,y)$ & $B^\flat \cong \bigoplus_{i \in \Z_3}T^i(X_{\beta})$\\
    $X_{\alpha} = \left(\begin{pmatrix}
        0&-y^2\\ 1&-y
        \end{pmatrix}, \begin{pmatrix}
        0&-y^3\\ 1&-y^2
        \end{pmatrix},\begin{pmatrix}
        0&-y^3\\ 1&-y
        \end{pmatrix}\right)$ & $A^\flat \cong \bigoplus_{i \in \Z_3}T^i(X_{\alpha})$\\
       $X_{\xi} =\left(\begin{pmatrix}
        y&0\\0&y^3
        \end{pmatrix}, \begin{pmatrix}
        0&y\\ y&1
        \end{pmatrix},\begin{pmatrix}
        -y&1\\y^2&0
        \end{pmatrix}\right)$ & $X^\flat \cong \bigoplus_{i \in \Z_3}T^i(X_{\xi}).$\\
         \bottomrule
\end{tabular}
 }
    \label{fig:1}
\end{figure}

\vspace{10pt}
    
The factorizations $\mc P_1, X_{\phi_1}, X_{\psi_1}, X_{\phi_2},$ and $X_{\beta}$ are each indecomposable since they are of size one. By \cite[Corollary 5.15]{tribone_matrix_2021}, the syzygy (and cosyzygy) of an indecomposable reduced matrix factorization is again indecomposable. Here a reduced matrix factorization means all the entries of all the matrices lie in the maximal ideal of $S$ (see Section \ref{section:reduced_dMFs}). Using \cite[2.11]{tribone_matrix_2021}, we have that $\Omega_{\MatFac{S}{3}{y^4}}^-(X_\beta) \cong X_\alpha$ and therefore, $X_\alpha$ is indecomposable.

Since $X_\xi$ is of size $2$, a non-trivial decomposition would be of the form $(y,b,c) \oplus (y^3,b',c')$ for some $b,c,b',c' \in S$. Since
$\det \left( \begin{smallmatrix} 0 & y \\ y & 1 \end{smallmatrix}\right)=-y^2$, the possibilities for $b$ and $b'$ are, up to units, $b=y^2$ and $b' = 1$ or $b=y=b'$. By considering cokernels, both cases lead to contradictions and so $X_\xi$ must be indecomposable.

By Proposition \ref{thm:flat_sharp_and_sharp_flat}, these seven factorizations, and each of their corresponding shifts, give the complete list of non-isomorphic indecomposable objects in $\MatFac{S}{3}{y^4}$ (21 in total).
\end{example}

To end this section, we discuss the relationship between the $2$-MF type of $f$ and the $d$-MF type of $f$ for $d > 2$. In one direction, we have the following consequence of Theorem \ref{thm:B}.

\begin{cor}\label{thm:finite_type_descends} Let $S$ be a regular local ring, $f$ a non-zero non-unit in $S$, and $d \ge 2$. If $f$ has finite $d$-MF type, then $f$ has finite $k$-MF type for all $2\leq k \leq d$. In particular, if $f$ has finite $d$-MF type for some $d\ge 2$, then $R=S/(f)$ has finite CM type.
\end{cor}

In general, the converse of Lemma \ref{thm:finite_type_descends} does not hold. The example below gives a polynomial of finite $2$-MF type but of infinite $3$-MF type.

\begin{example}\label{example:phi_abc}
Let $S=\mathbf{k}\llbracket x,y\rrbracket$ for an algebraically closed field with characteristic $\charac\mathbf{k}\neq 2,3$ and let $f = x^3 + y^3 \in S$. The hypersurface ring $R=S/(f)$ is a simple singularity of type $D_4$ and therefore has finite CM type.

Consider $R^\sharp = \mathbf{k}\llbracket x,y,z\rrbracket/(x^3+y^3+z^3)$, the $3$-fold branched cover of $R$. Following \cite{buchweitz_moore_2015}, to each point $(a,b,c)\in \mathbf{k}^3$ satisfying $a^3+b^3+c^3 = 0$ and $abc\neq 0$, we associate the \textit{Moore matrix}
\[X_{abc} = \begin{pmatrix}
ax &bz &cy  \\
by &cx &az \\
cz &ay &bx
\end{pmatrix}.\]
The $R^\sharp$-module $N_{abc} \coloneqq\cok(X_{abc})$ is MCM and is given by the matrix factorization $(X_{abc}, \frac{1}{abc}\text{adj}X_{abc}) \in \MatFac{S\llbracket z\rrbracket}{2}{x^3 + y^3 + z^3}$, where $\text{adj}X_{abc}$ is the classical adjoint of $X_{abc}$. Furthermore, $N_{abc}$ is indecomposable since $\det X_{abc}=abc(x^3+y^3+z^3)$ and $x^3+y^3+z^3$ is irreducible. Buchweitz and Pavlov give precise conditions for $X_{abc}$ to be matrix equivalent to $X_{a'b'c'}$ (see \cite[Proposition 2.13]{buchweitz_moore_2015}). In particular, their results imply that the collection $\{N_{abc}\}$, as $(a,b,c)$ varies over the curve $x^3+y^3+z^3$, gives an infinite collection of non-isomorphic indecomposable MCM $R^\sharp$-modules. It now follows from Theorem \ref{thm:finite_dMF_type_2} that $x^3 + y^3$ has infinite $3$-MF type.

With respect to the images of the standard basis on $S\llbracket z \rrbracket^3$, multiplication by $z$ on $N_{abc}$ is given by the $S$-matrix
\[\phi_{abc} = \begin{pmatrix}
0 &-\frac{c}{a}y &-\frac{a}{c}x\\
-\frac{c}{b}x &0 &-\frac{b}{c}y\\
-\frac{a}{b}y &-\frac{b}{a}x &0
\end{pmatrix}.\]
Therefore, we have that $N_{abc}^\flat = (\mu\phi_{abc},\mu\phi_{abc},\mu\phi_{abc}) \in \MatFac{S}{3}{x^3 +y^3}$, where $\mu^3=-1$. Hence, the collection of non-isomorphic indecomposable summands of $N_{abc}^\flat$, for all $(a,b,c)$ as above, forms an infinite collection of indecomposable objects in $\MatFac{S}{d}{f}$. Furthermore, the entries of $\phi_{abc}$ lie in the maximal ideal of $S$ so $x^3+y^3$ has infinite reduced $3$-MF type as well (see Section \ref{section:reduced_dMFs}).
\end{example}

\section{Decomposability of \texorpdfstring{$N^\flat$}{N-flat} and \texorpdfstring{$X^\sharp$}{X-sharp}} \label{section:decomp_sharp_flat}

Let $d\ge 2$ and $(S,\mf n, \mathbf{k})$ be a complete regular local ring. We maintain the same assumptions on $\mathbf{k}$ as in Section \ref{section:rep_type}, that is, we assume that $\mathbf{k}$ is algebraically closed of characteristic not dividing $d$. Let $f \in \mf n^2$ be non-zero, $R = S/(f)$, and $R^\sharp = S\llbracket z \rrbracket/(f+z^d)$.

Proposition \ref{thm:flat_sharp_and_sharp_flat} showed that both $N^{\flat\sharp}$ and $X^{\sharp\flat}$ decompose into a sum of $d$ objects. In this section we investigate the decomposability of $N^\flat$ and $X^\sharp$.

Recall that the shift functor $T: \MatFac{S}{d}{f} \to \MatFac{S}{d}{f}$ satisfies $T^d = 1_{\MatFac{S}{d}{f}}$. In particular, for any $X \in \MatFac{S}{d}{f}$, there exists a smallest integer $k\in\{1,2,\dots,d-1,d\}$ such that $T^kX \cong X$. We call $k$ the \textit{order} of $X$.

\begin{lem}\label{thm:order_of_X}
For any $X \in \MatFac{S}{d}{f}$, the order of $X$ is a divisor of $d$.
\end{lem}

\begin{proof}
For a given $X \in \MatFac{S}{d}{f}$, the cyclic group of order $d$ generated by $T$ acts on the set of equivalence classes $\{[T^iX]:i \in \Z_d\}$. In particular, the stabilizer of $[X]$ is generated by $T^k$ for some $k\mid d$ which can be taken to be the smallest possible in $\{1,2,\dots,d\}$. It follows that the order of $X$ is $k$.
\end{proof}

The next result builds on an idea of Kn\"orrer \cite[Lemma 1.3]{knorrer_cohen-macaulay_1987} and Gabriel \cite[p. 95]{gabriel_universal_1981}. The proof is based on \cite[Lemma 8.25]{leuschke_cohen-macaulay_2012} which states that a matrix factorization $(\phi,\psi)\in \MatFac{S}{2}{f}$ satisfying $(\phi,\psi)\cong (\psi,\phi)$ is isomorphic to a factorization of the form $(\phi_0,\phi_0)$\footnote{The proof of \cite[Lemma 8.25]{leuschke_cohen-macaulay_2012} contains a small typo which we remedy in the proof of Proposition \ref{thm:semi_sym_dMF}. Specifically, the morphism $(\alpha,\beta)$ should be replaced by $(c^{-1/2}\alpha,c^{-1/2}\beta)$.}. For $d>2$, the situation is similar, but the divisors of $d$ play a role. Specifically, if $X$ has order $k$, then $X$ is isomorphic to the concatenation of $k$ matrices, repeated $d/k$ times.

\begin{prop}\label{thm:semi_sym_dMF}
Let $X \in \MatFac{S}{d}{f}$ be indecomposable of size $n$ and assume $X$ has order $k < d$. Then there exist $S$-homomorphisms $\phi_1',\phi_2',\dots,\phi_k'$ such that $(\phi_1'\phi_2'\cdots\phi_k')^{\frac{d}{k}} = f\cdot I_n$ and
\[X \cong (\phi_1',\phi_2',\dots,\phi_k', \phi_1',\phi_2',\dots,\phi_k', \dots,\phi_1',\phi_2',\dots,\phi_k').\]
\end{prop}

\begin{proof}
Let $X = (\phi_1:F_2\to F_1,\phi_2:F_3 \to F_2,\dots, \phi_d:F_1 \to F_d)$ and set $r = d/k$. By assumption, there is an isomorphism $\alpha = (\alpha_1,\dots,\alpha_d): X \to T^kX$. By applying $T^k(-)$ repeatedly, we obtain an automorphism $\tilde \alpha$ of $X$ defined by the composition
\[\begin{tikzcd}[column sep = 3.5em]
X \rar{\alpha} &T^kX \rar{T^k(\alpha)} &T^{2k}X \rar{T^{2k}(\alpha)} &\cdots \rar{T^{(r-2)k}(\alpha)} &T^{(r-1)k}X \rar{T^{(r-1)k}(\alpha)} &X.
\end{tikzcd}\] In particular,
\(\tilde\alpha\) is the $d$-tuple \((\alpha_{i+(r-1)k} \alpha_{i + (r-2)k} \cdots \alpha_{i+k} \alpha_i)_{i=1}^d.\) Since $X$ is indecomposable the endomorphism ring $\Lambda \coloneqq \End_{\MatFac{S}{d}{f}}(X)$ is local. Since $\mathbf{k}$ is algebraically closed, it cannot have any non-trivial finite extensions which are division rings. Hence the division ring $\Lambda / \rad\Lambda$ must be isomorphic to  $\mathbf{k}$. This allows us to write
\(\tilde \alpha = c\cdot 1_X + \rho\) for some $c \in \mathbf{k}^\times$ and $\rho \in \rad \Lambda$. Since $\charac \mathbf{k}\not\vert\,\, d$, we may scale $\alpha$ by $c^{-\frac{1}{r}}$ and assume $\tilde \alpha = 1_X + \rho$ for $\rho = (\rho_1,\rho_2,\dots,\rho_d) \in \rad\Lambda$.

If $i \in \Z_d$, then
\[\begin{split}
    \alpha_i\rho_i &= \alpha_i(\alpha_{i+(r-1)k}\alpha_{i+(r-2)k}\cdots\alpha_{i+k}\alpha_i - 1_{F_i})\\
                   &= (\alpha_i\alpha_{i+(r-1)k}\alpha_{i+(r-2)k}\cdots\alpha_{i+k} - 1_{F_{i+k}})\alpha_i\\
                   &=\rho_{i+k}\alpha_i.
\end{split}\] Represent the function $g(x) = (1+x)^{-1/r}$ by its Maclaurin series and define, for each $i \in \Z_d$,
\[\beta_i \coloneqq \alpha_ig(\rho_i) = g(\rho_{i+k})\alpha_i: F_i \to F_{i + k}.\]
For $i \in \Z_d$, we have that
\[\beta_i\phi_i = g(\rho_{i+k})\alpha_i\phi_i = g(\rho_{i+k})\phi_{i+k}\alpha_{i+1} = \phi_{i+k}g(\rho_{i+k+1})\alpha_{i+1} = \phi_{i+k}\beta_{i+1}.\] Hence, $\beta = (\beta_1,\beta_2,\dots,\beta_d) \in \Hom_{\MatFac{S}{d}{f}}(X,T^kX)$. By repeatedly applying $\alpha_ig(\rho_i)=g(\rho_{i+k})\alpha_i$, we have that
\[\begin{split}
    \beta_i\beta_{i-k}\beta_{i-2k}\cdots\beta_{i+2k}\beta_{i+k} &= (\alpha_{i}g(\rho_{i}))(\alpha_{i-k}g(\rho_{i-k}))(\alpha_{i-2k}g(\rho_{i-2k}))\cdots(\alpha_{i+k}g(\rho_{i+k}))\\
    &=\alpha_i\alpha_{i-k}\cdots\alpha_{i+2k}\alpha_{i+k}g(\rho_{i+k})^r\\
    &=(1_{F_{i+k}} + \rho_{i+k})(1_{F_{i+k}} + \rho_{i+k})^{-1}\\
    &=1_{F_{i+k}}.
\end{split}\]
Hence, $\beta_i$ is an isomorphism for each $i \in \Z_d$ and therefore the morphism $\beta$ is an isomorphism of matrix factorizations.

We claim that $X \cong (\beta_1\phi_1,\dots,\phi_k,\beta_1\phi_1,\dots,\phi_k,\dots,\beta_1\phi_1,\dots,\phi_k)$. For $0\leq j \leq r-1$ and $2\leq t \leq k+1$, define $\gamma_{j,t}$ to be the composition of the homomorphisms $\beta_i$ beginning at $F_{t+jk}$ of length $r-j$. In other words,
\[\gamma_{j,k} = \beta_{t+(r-1)k}\beta_{t+(r-2)k}\cdots\beta_{t+(j+1)k}\beta_{t + jk}: F_{t+jk} \to F_t.\] Note that each $\gamma_{j,k}$ is an isomorphism. For $j=0$, the index $-1$ is interpreted as $r-1$ so that $\gamma_{-1,k+1} = \beta_1$. 

Let $0\leq j \leq r-1$ and $2\leq t \leq k+1$. To finish the proof, it suffices to show that the following diagram commutes:
\[\begin{tikzcd}[column sep=3.5em]
F_{k+1 + jk} \dar{\gamma_{j,k+1}}\rar{\phi_{k+jk}} &F_{k + jk} \dar{\gamma_{j,k}} \rar{\phi_{k-1 + jk}} &\cdots \rar{\phi_{2 + jk}} &F_{2+jk} \rar{\phi_{1 + jk}} \dar{\gamma_{j,2}} &F_{1 + jk}\dar{\gamma_{j-1,k+1}}\\
F_{k+1} \rar{\phi_k}      &F_{k} \rar{\phi_{k-1}} &\cdots \rar{\phi_2} &F_2 \rar{\beta_1\phi_1} &F_{k+1}.
\end{tikzcd}\] The commutativity can be broken into three steps. First, we show that $\gamma_{j-1,k+1}\phi_{1+jk} = \beta_1\phi_1\gamma_{j,2}.$ Repeatedly applying $\beta_i\phi_i = \phi_{i+k}\beta_{i+1}$, $i \in \Z_d$, we have that
\[\begin{split}
    \gamma_{j-1,k+1}\phi_{1+jk}
    &=\beta_{1}\beta_{1-k}\beta_{1-2k}\cdots\beta_{1+(j+1)k}\beta_{1+jk}\phi_{1+jk}\\
    &=\beta_1\phi_1\beta_{2-k}\beta_{2-2k}\cdots\beta_{2+(j+1)k}\beta_{2+jk}\\
    &=\beta_1\phi_1\gamma_{j,2}.
\end{split}\]
Similarly, for $2\leq t \leq k$, we have that
\[\begin{split}
    \gamma_{j,t}\phi_{t+jk} &= \beta_{t-k}\beta_{t-2k}\cdots\beta_{t+(j+1)k}\beta_{t + jk}\phi_{t+jk}\\
    &=\phi_t\beta_{t+1 - k}\beta_{t+1 -2k}\cdots\beta_{t+1 +(j+1)k}\beta_{t+1 +jk}\\
    &=\phi_t\gamma_{j,t+1}
\end{split}\]
and
\[\begin{split}
    \gamma_{j,k}\phi_{k + jk} &= \beta_d\beta_{-k}\beta_{-2k}\cdots \beta_{(j+1)k}\phi_{(j+1)k}\\
    &=\phi_k\beta_1\beta_{1-k}\cdots\beta_{1+(j+1)k}\\
    &=\phi_k\gamma_{j,k+1}.
\end{split}\]
Thus, the $d$-tuple $\gamma = (\gamma_{-1,k+1},\gamma_{0,2},\gamma_{0,3},\dots,\gamma_{0,k+1},\gamma_{1,2},\dots,\gamma_{r-1,k})$ forms an isomorphism from $X$ to $(\beta_1\phi_1,\dots,\phi_k,\beta_1\phi_1,\dots,\phi_k,\dots,\beta_1\phi_1,\dots,\phi_k)$.
\end{proof}

The special case of order $1$ will be important going forward.

\begin{cor}\label{thm:sym_dMF}
Let $X \in \MatFac{S}{d}{f}$ be indecomposable of size $n$ and assume that $X \cong TX$. Then there exists a homomorphism $\phi:S^n \to S^n$ such that $\phi^d = f \cdot I_n$ and $X \cong (\phi,\phi,\dots,\phi)$. \qed
\end{cor}

\begin{prop} \label{thm:decomp_thm_setup} Let $X \in \MatFac{S}{d}{f}$, $N \in \MCM(R^\sharp)$, and assume both $X$ and $N$ are indecomposable objects.
\begin{enumerate}[label = (\roman*)]
    \item \label{thm:decomp_thm_setup_1} If $X \cong TX$, then $X \cong M^\flat$ for some $M \in \MCM(R^\sharp)$.
    \item \label{thm:decomp_thm_setup_2} If $N \cong \sigma^*N$, then $N \cong Y^\sharp$ for some $Y \in \MatFac{S}{d}{f}$.
\end{enumerate}
\end{prop}

\begin{proof}
If $X \cong TX$, then Corollary \ref{thm:sym_dMF} implies that there exists a free $S$-module $F$ and an endomorphism $\phi:F \to F$ such that $\phi^d = f\cdot 1_F$ and $X \cong (\phi,\phi,\dots,\phi) \in \MatFac{S}{d}{f}$. The pair $(F,\mu^{-1}\phi)$ defines an $\MCM(R^\sharp)$ module $M$ as follows: As an $S$-module, $M = F$, and the $z$-action on $M$ is given by $z\cdot m = \mu^{-1}\phi(m)$ for all $m \in M$, where $\mu \in S$ satisfies $\mu^d = -1$. Since $(\mu^{-1}\phi)^d= -f\cdot 1_M$, $M$ is naturally an $R^\sharp$-module. Since $M=F$ is free over $S$, it is MCM over $R^\sharp$. By applying $(-)^\flat$, we have that $M^\flat = (\phi,\phi,\dots,\phi)\cong X$.

Assume $N \cong \sigma^*N$. Using a similar technique to the proof of Proposition \ref{thm:semi_sym_dMF}, we obtain an isomorphism of $R^\sharp$-modules $\theta: N \to \sigma^* N$ such that 
\[(\sigma^{d-1})^*\theta\circ(\sigma^{d-2})^*\theta\circ\cdots\circ\sigma^*\theta\circ\theta = 1_{N}.\] Such an isomorphism defines the structure of an $R^\sharp[\sigma]$-module on $N$. Thus, by \cite[Theorem 4.4]{tribone_matrix_2021}, there exists $Y \in \MatFac{S}{d}{f}$ such that $\mc B(Y) \cong N$ as $R^\sharp[\sigma]$-modules and therefore $Y^\sharp \cong N$ as $R^\sharp$-modules by \ref{thm:A_B_flat_sharp}\ref{thm:A_B_flat_sharp_1}.
\end{proof}

\begin{prop} \label{thm:decomp_thm} Let $X$ be indecomposable in $\MatFac{S}{d}{f}$ and $N$ be indecomposable in $\MCM(R^\sharp)$.
\begin{enumerate}[label = (\roman*)]
    \item \label{thm:decomp_thm_1} Assume $X \cong TX$. Then $X^\sharp \cong \bigoplus_{k \in \Z_d} (\sigma^k)^*M$ for some indecomposable $M \in \MCM(R^\sharp)$ such that $M \not\cong \sigma^*M$.
    \item \label{thm:decomp_thm_2} The number of indecomposable summands of $X^\sharp$ is at most $d$. Furthermore, if $X^\sharp$ has exactly $d$ indecomposable summands, then $X \cong TX$.
    \item \label{thm:decomp_thm_3} Assume $N \cong \sigma^*N$. Then $N^\flat \cong \bigoplus_{k \in \Z_d} T^kY$ for some indecomposable $Y \in \MatFac{S}{d}{f}$ such that $Y \not\cong TY$.
    \item \label{thm:decomp_thm_4} The number of indecomposable summands of $N^\flat$ is at most $d$. Furthermore, if $N^\flat$ has exactly $d$ indecomposable summands, then $N \cong \sigma^*N$.
\end{enumerate}
\end{prop}

\begin{proof}
If $X \cong TX$, then \ref{thm:decomp_thm_setup}\ref{thm:decomp_thm_setup_1} implies that $X\cong M^\flat$ for some $M \in \MCM(R^\sharp)$. By Proposition \ref{thm:flat_sharp_and_sharp_flat}, we have that $X^\sharp \cong M^{\flat\sharp} \cong \bigoplus_{k \in \Z_d} (\sigma^k)^*M$. Similarly, if $N \cong \sigma^*N$, then \ref{thm:decomp_thm_setup}\ref{thm:decomp_thm_setup_2} and Proposition $\ref{thm:flat_sharp_and_sharp_flat}$ imply that $N^\flat \cong \bigoplus_{k \in \Z_d} T^kY$ for some $Y \in \MatFac{S}{d}{f}$. 

Next, in the case that $TX \cong X$, we show that $M$ above is indecomposable and satisfies $M \not\cong \sigma^*M$.
\begin{itemize}
    \item Suppose $M \cong M_1\oplus M_2$ for non-zero $M_1,M_2 \in \MCM(R^\sharp)$. Then $(\sigma^k)^*M \cong (\sigma^k)^*M_1 \oplus (\sigma^k)^*M_2$ for each $k \in \Z_d.$ Therefore,
    \[X^d \cong X^{\sharp\flat} \cong \bigoplus_{k\in \Z_d}((\sigma^k)^*M_1)^\flat \oplus ((\sigma^k)^*M_2)^\flat .\] This contradicts KRS since the left side has precisely $d$ indecomposable summands while the right hand side has at least $2d$ indecomposable summands. Hence, $M$ is indecomposable.
    \item Suppose that $\sigma^*M \cong M$. Then, since $M$ is indecomposable, the arguments above imply that $M^\flat$ decomposes into a sum of at least $d$ indecomposable summands. Since $T(M^\flat) \cong M^\flat$, we have
    \[X^d \cong X^{\sharp\flat} \cong (M^\flat)^{\sharp\flat} \cong (M^\flat)^d.\] Since $X$ is indecomposable, the left hand side has precisely $d$ indecomposable summands while the right hand side has at least $d^2$ indecomposable summands. Once again, we have a contradiction and so $M \not\cong \sigma^*M$.
\end{itemize}
This completes the proof of \ref{thm:decomp_thm_1}. We omit the remaining assertions from \ref{thm:decomp_thm_3} as they follow similarly.

In order to prove \ref{thm:decomp_thm_2}, suppose $X^\sharp = M_1\oplus M_2 \oplus \cdots \oplus M_t$ for non-zero $M_i \in \MCM(R^\sharp)$. Then 
\begin{equation}\label{equation:decomp_thm}X \oplus TX \oplus \cdots \oplus T^{d-1}X \cong X^{\sharp\flat} \cong M_1^\flat \oplus M_2^\flat \oplus \cdots \oplus M_t^\flat.\end{equation} The left hand side has precisely $d$ indecomposable summands and therefore $t \leq d$.

If $X^\sharp$ decomposes into exactly $d$ indecomposables, that is, if $t=d$, then \eqref{equation:decomp_thm} implies that $M_i^\flat$ is indecomposable for each $i$ and that $X \cong M_j^\flat$ for some $1\leq j \leq d$. Then
\(TX \cong T(M_j^\flat) = M_j^\flat \cong X.\) 

The proof of \ref{thm:decomp_thm_4} is similar, observing that $\sigma^*(X^\sharp) \cong X^\sharp$ for any $X \in \MatFac{S}{d}{f}$.

\end{proof}

\section{Reduced matrix factorizations and Ulrich modules}\label{section:reduced_dMFs}

Let $(S,\mf n, \mathbf{k})$ be a complete regular local ring, $0\neq f \in \mf n^2$, and let $d \ge 2$ be an integer. Assume $\mathbf{k}$ is algebraically closed of characteristic not dividing $d$. In this section, we will consider the following special class of matrix factorizations in $\MatFac{S}{d}{f}$.

\begin{defi}
A matrix factorization $X = (\phi_1,\phi_2,\dots,\phi_d) \in \MatFac{S}{d}{f}$ is called \textit{reduced} if $\phi_k:F_{k+1}\to F_k$ is minimal for each $k \in \Z_d$, that is, if $\Ima \phi_k \subseteq \mf n F_k$. Equivalently, after choosing bases, $X$ is reduced if and only if the entries of $\phi_k$ lie in $\mf n$ for all $k \in \Z_d$. We say that $f$ has \textit{finite reduced $d$-MF type} if there are, up to isomorphism, only finitely many indecomposable reduced matrix factorizations $X \in \MatFac{S}{d}{f}$.
\end{defi}

In the case $d=2$, any indecomposable non-reduced matrix factorization is isomorphic to either $(1,f)$ or $(f,1)$ in $\MatFac{S}{2}{f}$ \cite[Remark 7.5]{yoshino_cohen-macaulay_1990}. In particular, this implies that finite $2$-MF type is equivalent to finite reduced $2$-MF type.

For $d > 2$, the situation is quite different. There at least as many non-reduced indecomposable $d$-fold factorizations of $f$ as there are reduced ones \cite[Corollary 5.15]{tribone_matrix_2021}. Moreover, finite $d$-MF type clearly implies finite reduced $d$-MF type but the converse does not hold for $d>2$ as we will show in Example \ref{example:x^a_y^a+i}.

\begin{defi}\,
\begin{enumerate}[label = (\roman*)]
    \item Let $X = (\phi_1,\dots,\phi_d) \in \MatFac{S}{d}{f}$ and pick bases to consider $\phi_k$, $k \in \Z_d$, as a square matrix with entries in $S$. Following \cite{buchweitz_cohen-macaulay_1987}, we define $I(\phi_k)$ to be the ideal generated by the entries of $\phi_k$ and set $I(X) = \sum_{k\in \Z_d}I(\phi_k)$. Note that the ideal $I(X)$ does not depend on the choice of bases.
    \item Let $c_d(f)$ denote the collection of proper ideals $I$ of $S$ such that $f \in I^d$.
\end{enumerate}
\end{defi}

In the case $d=2$, Theorem \ref{thm:BGS_Knorrer} implies that reduced $2$-MF type is determined by the cardinality of the set $c_2(f)$. One implication of Theorem \ref{thm:BGS_Knorrer} is proven explicitly in  \cite{buchweitz_cohen-macaulay_1987}. The authors show that the association $X \mapsto I(X)$ forms a surjection from the set of isomorphism classes of reduced $2$-fold matrix factorizations of $f$ onto the set $c_2(f)$. Hence, if there are only finitely many indecomposable reduced $2$-fold matrix factorizations of $f$ up to isomorphism, then the set $c_2(f)$ is finite.

The following result of Herzog, Ulrich, and Backelin shows that the association $X \mapsto I(X)$ remains surjective in the case $d > 2$.

\begin{thm}[\cite{herzog_linear_1991}, Theorem 1.2]
Let $I$ be a proper ideal of $S$ and $d \ge 2$. If $f \in I^d$, then there exists a reduced matrix factorization $X \in \MatFac{S}{d}{f}$ such that $I(X) = I$. \qed
\end{thm}

\begin{cor}\label{thm:finite_red_implies_c_d(f)_finite}
Suppose $f$ has finite reduced $d$-MF type. Then $c_d(f)$ is a finite collection of ideals of $S$. \qed
\end{cor}

Corollary \ref{thm:finite_red_implies_c_d(f)_finite} extends one direction of Theorem \ref{thm:BGS_Knorrer}; however, the converse does not hold for $d > 2$ as shown by the next example.

\begin{example}\label{example:D_infty}
Let $S=\mathbf{k}\llbracket x, y\rrbracket$ with $\charac \mathbf{k} \neq 2$ and $f=x^2y \in S$. Then the one-dimensional $D_\infty$ singularity $R=S/(f)$ has countably infinite CM type by \cite[Proposition 4.2]{buchweitz_cohen-macaulay_1987}. For each $k\ge 1$, we have a reduced matrix factorization of $x^2y$ with $3$ factors:
\[X_k = \left(\begin{pmatrix}
x &y^{k}\\
0 &-x
\end{pmatrix}, \begin{pmatrix}
y &0\\
0 &y
\end{pmatrix}, \begin{pmatrix}
x &y^k\\
0 &-x
\end{pmatrix}\right) \in \MatFac{S}{3}{x^2y}.\] Any isomorphism $X_k \to X_j$ for $k,j\ge 1$ induces an isomorphism  of $R$-modules $\cok\begin{pmatrix}
x &y^{k}\\
0 &-x
\end{pmatrix} \to \cok\begin{pmatrix}
x &y^{j}\\
0 &-x
\end{pmatrix}$. Such an isomorphism is only possible if $k=j$, that is, $X_k \cong X_j$ if and only if $k=j$. Since $X_k$ is reduced and the MCM $R$-module $\cok\begin{pmatrix}
x &y^{k}\\
0 &-x
\end{pmatrix}$ is indecomposable, \cite[Lemma 6.6]{tribone_matrix_2021} implies that $X_k$ is indecomposable. Thus, $x^2y$ has infinite reduced $3$-MF type.

On the other hand, it's not hard to see that $c_3(x^2y)$ contains only the maximal ideal. So $c_3(x^2y)$ is a finite set but $x^2y$ has infinite reduced $3$-MF type.
\end{example}

Let $N$ be an MCM $R^\sharp$-module and let $\mu_{R^\sharp}(N)$ denote the minimal number of generators of $N$. Recall that $N$ is finitely generated and free over $S$. We will see below that there is an inequality \begin{equation}\label{equation:mu_rank}
    \mu_{R^\sharp}(N) \leq \rank_S(N).
\end{equation} In the following, we consider MCM $R^\sharp$-modules $N$ where the equality $\mu_{R^\sharp}(N) = \rank_S(N)$ is attained.

As we saw in Example \ref{example:E_6}, a matrix factorization of the form $N^\flat$, obtained by computing multiplication by $z$ on an MCM $R^\sharp$-module $N$, can be non-reduced. We will show below that the matrix representing multiplication by $z$ on $N$ contains unit entries precisely when $\mu_{R^\sharp}(N) < \rank_S(N)$. In other words, the restriction of the functor $(-)^\flat: \MCM(R^\sharp) \to \MatFac{S}{d}{f}$ to the subcategory of MCM $R^\sharp$-modules satisfying $\mu_{R^\sharp}(N) = \rank_S(N)$ produces only reduced matrix factorizations of $f$ with $d$ factors. Conversely, the image of the functor $(-)^\sharp: \MatFac{S}{d}{f} \to \MCM(R^\sharp)$, restricted to the subcategory of reduced matrix factorizations of $f$, consists exactly of the MCM $R^\sharp$-modules $N$ satisfying $\mu_{R^\sharp}(N) = \rank_S(N)$.

\begin{lem}\label{thm:rank_S(N)}
Let $N$ be an MCM $R^\sharp$-module and assume that $f+z^d$ is irreducible. Then $N$ is a finitely generated free $S$-module satisfying
\[\mu_{R^\sharp}(N) \leq \rank_S(N) = d\cdot \rank_{R^\sharp}(N) = \rank_{S}(R^\sharp)\cdot \rank_{R^\sharp}(N).\]
\end{lem}

\begin{proof}
Let $(\Phi : S\llbracket z \rrbracket^n \to S\llbracket z \rrbracket^n,\Psi: S\llbracket z \rrbracket^n \to S\llbracket z \rrbracket^n) \in \MatFac{S\llbracket z \rrbracket}{2}{f+z^d}$ be a matrix factorization of $f + z^d$ such that $\Phi$ is minimal and $\cok\Phi = N$. Since $\Phi$ is minimal, $n = \mu_{R^\sharp}(N)$. Then $\det \Phi = u(f+z^d)^k$ for some $1\leq k \leq n$ and some unit $u \in S\llbracket z \rrbracket$. Recall that $k = \rank_{R^\sharp}(N)$ by \cite[Propoistion 5.6]{eisenbud_homological_1980}. By tensoring with $S=S\llbracket z \rrbracket/(z)$, we find that $\det \bar \Phi = v \cdot f^k$, where $\bar \Phi = \Phi\otimes_{S\llbracket z \rrbracket}  1_S$ and $v \in S$ is a unit. Moreover, $\bar \Phi$ is injective since $\bar\Phi \bar\Psi = f\cdot 1_{S^n} = \bar\Psi \bar\Phi$, and we have a minimal presentation of $N/zN$ over $S$:
\[\begin{tikzcd}
0 \rar &S^n \rar{\bar \Phi} &S^n \rar &N/zN \rar &0.
\end{tikzcd}\]

On the other hand, since $N$ is MCM over $R^\sharp$, it is finitely generated and free as an $S$-module. Let $r = \rank_S(N)$ and consider the map $\phi:S^r \to S^r$ representing multiplication by $z$ on $N$. This map also gives a presentation of $N/zN$ over $S$, though the presentation may not be minimal (see Example \ref{example:E_6}). Thus, there exists a commutative diagram with vertical isomorphisms
\begin{equation}\label{equation:presentation_of_N/zN}
\begin{tikzcd}[ampersand replacement=\&,column sep = large]
0 \rar \& S^r \dar \rar{\begin{pmatrix}
\bar \Phi &0\\ 0 &I_{r-n}
\end{pmatrix}} \&S^r \dar \rar \&N/zN \dar[equals] \rar \&0\\
0 \rar \& S^r \rar{\phi} \&S^r \rar \&N/zN \rar \&0.
\end{tikzcd}\end{equation}
This implies that $\mu_R(N/zN) \leq r$, where $R = S/(f)$ as usual. The desired inequality now follows from the fact that $\mu_R(N/zN) = \mu_{R^\sharp}(N)$. Furthermore, the diagram implies that $\det \phi = v' \cdot f^k$ for some unit $v'$. However, since $\phi^d = -f\cdot I_r$, we have that, up to units, 
\(f^r = (\det\phi)^d = f^{kd}.\)
Thus,
\(\rank_S(N) = r = dk = d\cdot \rank_{R^\sharp}(N).\)
\end{proof}

\begin{lem}\label{thm:reduced_MFs_vs_rank_S}
Assume $f+z^d$ is irreducible. Let $N$ be an MCM $R^\sharp$-module and let $X \in \MatFac{S}{d}{f}$. Then $\mu_{R^\sharp}(N) = \rank_S(N)$ if and only if $N^\flat \in \MatFac{S}{d}{f}$ is reduced, and $X^\sharp$ satisfies $\mu_{R^\sharp}(X^\sharp)=\rank_S(X^\sharp)$ if and only if $X$ is reduced.
\end{lem}

\begin{proof}
Let $N \in \MCM(R^\sharp)$ and set $r = \rank_S(N)$. Let $\phi:S^r \to S^r$ be the $S$-linear map representing multiplication by $z$ on $N$. Then the presentation of $N/zN$ given by $\phi$ in \eqref{equation:presentation_of_N/zN} is minimal if and only if $r=\rank_S(N) = \mu_R(N/zN)$, where $R=S/(f)$. Since $\mu_R(N/zN) = \mu_{R^\sharp}(N)$, we have that $\phi$ is minimal if and only if $\rank_S(N) = \mu_{R^\sharp}(N)$. This proves the first statement since $N^\flat = (\mu\phi,\mu\phi,\cdots,\mu\phi) \in \MatFac{S}{d}{f}$.

By Proposition \ref{thm:flat_sharp_and_sharp_flat}, $X^{\sharp\flat} \cong \bigoplus_{k \in \Z_d}T^kX$ which is reduced if and only if $X$ is reduced. The second statement now follows from the first by taking $N = X^\sharp \in \MCM(R^\sharp)$.
\end{proof}

Lemma \ref{thm:reduced_MFs_vs_rank_S} gives us a specialization of Corollary \ref{thm:cor_knorrer} and Theorem \ref{thm:finite_dMF_type_2}.

\begin{prop} \label{thm:finite_reduced_type} Assume $f + z^d$ is irreducible.
\begin{enumerate}[label = (\roman*)]
    \item \label{thm:finite_reduced_type_1} For any reduced $X \in \MatFac{S}{d}{f}$, there exists $N \in \MCM(R^\sharp)$ satisfying $\rank_S(N) = \mu_{R^\sharp}(N)$ such that $X$ is isomorphic to a direct summand of $N^\flat$.
    \item \label{thm:finite_reduced_type_2} For any $N \in \MCM(R^\sharp)$ satisfying $\rank_S(N) = \mu_{R^\sharp}(N)$, there exists reduced $X \in \MatFac{S}{d}{f}$ such that $N$ is isomorphic to a direct summand of $X^\sharp$.
\end{enumerate}

In particular, $f$ has finite reduced $d$-MF type if and only if there are, up to isomorphism, only finitely many indecomposable MCM $R^\sharp$-modules $N$ satisfying $\rank_S(N) = \mu_{R^\sharp}(N)$.
\end{prop}

\begin{proof}
Both \ref{thm:finite_reduced_type_1} and \ref{thm:finite_reduced_type_2} follow from Lemma \ref{thm:reduced_MFs_vs_rank_S} and Proposition \ref{thm:flat_sharp_and_sharp_flat}. The final statement follows as in the proof of Theorem \ref{thm:finite_dMF_type_2} by noticing that a matrix factorization $Y \in \MatFac{S}{d}{f}$ is reduced if and only if every summand of $Y$ is reduced and that an MCM $R^\sharp$-module $N$ satisfies $\mu_{R^\sharp}(N) = \rank_S(N)$ if and only if every summand of $N$ satisfies the same equality.
\end{proof}

For a module $M$ over a local ring $A$, we let $e(M)$ denote the multiplicity of $M$. If $M$ is an MCM $A$-module, there is a well known inequality $\mu_A(M) \leq e(M)$. The class of MCM modules satisfying $\mu_A(M) = e(M)$ are called \textit{Ulrich modules}. For background on Ulrich modules we refer the reader to 
\cite{beauville_introduction_2018}, \cite{brennan_maximally_1987}, \cite{herzog_maximal_1987}, and \cite{herzog_linear_1991}. If $A$ is a domain, then we may compute the multiplicity of $M$ as $e(M) = e(A)\cdot \rank_A(M)$ (see \cite[Proposistion 1.6]{yoshino_cohen-macaulay_1990}).

In the case of the $d$-fold branched cover of $R$, we have the following connection between reduced $d$-fold matrix factorizations of $f$ and Ulrich modules over $R^\sharp$.

\begin{cor}\label{thm:Ulrich}
Assume $d \leq \textup{ord}(f)$ and that $f+z^d$ is irreducible. Let $N \in \MCM(R^\sharp)$. Then $N$ is an Ulrich $R^\sharp$-module if and only if $N^\flat \in \MatFac{S}{d}{f}$ is a reduced matrix factorization.

In particular, $f$ has finite reduced $d$-MF type if and only if there are, up to isomorphism, only finitely many indecomposable Ulrich $R^\sharp$-modules.
\end{cor}

\begin{proof}
Since $d \leq \textup{ord}(f)$, the multiplicity of $R^\sharp = S\llbracket z \rrbracket/(f+z^d)$ is $d$. Hence, an MCM $R^\sharp$-module $N$ is Ulrich if and only if $\mu_{R^\sharp}(N) = d \cdot \rank_{R^\sharp}(N)$. By Lemma \ref{thm:rank_S(N)}, the quantity $d\cdot \rank_{R^\sharp}(N)$ is equal to the rank of $N$ as a free $S$-module. Thus, $N$ is Ulrich if and only if $\mu_{R^\sharp}(N) = \rank_S(N)$. Both statements now follow from Proposition \ref{thm:finite_reduced_type}.
\end{proof}

\begin{remark} Let $d=2$ so that $R^\sharp = S\llbracket z \rrbracket/(f+z^2)$ is the double branched cover. The condition $\rank_S(N) = \mu_{R^\sharp}(N)$ is redundant in this case. An MCM $R^\sharp$-module $N$ satisfies $\rank_S(N) = \mu_{R^\sharp}(N)$ if and only if $N$ has no summands isomorphic to $R^\sharp$ (this follows from the proof of \cite[Lemma 8.17 \textit{(iii)}]{leuschke_cohen-macaulay_2012}). In other words, the conclusion of Proposition \ref{thm:finite_reduced_type} is simply a restatement of Kn\"orrer's Theorem (Theorem \ref{thm:knorrer_4_cats} above) when $d=2$. Furthermore, Corollary \ref{thm:Ulrich} implies that any MCM $R^\sharp$-module with no free summands is an Ulrich module. This is a known result of Herzog-K\"uhl \cite[Corollary 1.4]{herzog_maximal_1987} since the multiplicity of $R^\sharp$ is two.
\end{remark}

\begin{example}\label{example:x^a_y^a+i}
Let $\mathbf{k}$ be an algebraically closed field of characteristic zero and consider the one-dimensional hypersurface ring \vspace{-4pt}\[R_{a,i} = \mathbf{k}\llbracket x ,y \rrbracket /(x^a + y^{a+i}), \quad a\ge 2, i \ge 0.\]
If $i=1$ or $i=2$, then, by \cite[Theorem A.3]{herzog_linear_1991}, $R_{a,i}$ has only finitely many isomorphism classes of indecomposable Ulrich modules. Set $S=\mathbf{k}\llbracket y \rrbracket$ and consider $R_{a,i}$ as the $a$-fold branched cover of $R=\mathbf{k}\llbracket y\rrbracket/(y^{a+i})$. Since $e(R_{a,i}) = a$, Corollary \ref{thm:Ulrich} implies that $y^{a+i},$ for $i \in \{1,2\},$ has only finitely many isomorphism classes of reduced indecomposable $a$-fold matrix factorizations. In other words, $y^{a+i}$ has finite reduced $a$-MF type for $i=1,2$ and any $a\ge 2$.

The methods in \cite[Theorem A.3]{herzog_linear_1991} can be used to compute the isomorphism classes of indecomposable reduced matrix factorizations of $y^{a+i}$. For instance, let $a \ge 2$ and $i=1$. Then $R_{a,1} \cong \mathbf{k}\llbracket t^a,t^{a+1}\rrbracket$ and $t^a$ is a minimal reduction of the maximal ideal $\mf m$ of $R_{a,1}$. Hence, $R_{a,1}' = R_{a,1}[\{\frac{r}{t^a}: r \in \mf m\}]  = \mathbf{k}\llbracket t \rrbracket$ is the first quadratic transform of $R_{a,1}$. By \cite[Corollary A.1]{herzog_linear_1991}, an $R_{a,1}$-module $M$ is Ulrich over $R_{a,1}$ if and only if it is MCM over $R_{a,1}'$. Since $R_{a,1}' = \mathbf{k}\llbracket t \rrbracket$ is a regular local ring, the only indecomposable MCM $R_{a,1}'$-module is $R_{a,1}'$ itself.

As an $S\cong \mathbf{k}\llbracket t^a \rrbracket$-module, $R_{a,1}' = \mathbf{k}\llbracket t \rrbracket$ is free with basis given by $\{1,t,t^2,\dots,t^{a-1}\}$. Thus, multiplication by $x=t^{a+1}$ on the basis $\{1,t,\dots,t^{a-1}\}$ is given by the mapping
\[t^k \mapsto t^{a+1+k} = t^at^{k+1}\] for $0\leq k \leq a-1$. Since $y=t^a$, it follows that multiplication by $x$ on the MCM $R_{a,1}$-module $R_{a,1}'$ is given by the $a\times a$ matrix with entries in $\mathbf{k}\llbracket y \rrbracket$
\[\phi = \begin{pmatrix}
0 &0 &\cdots &0 &y^2\\
y &0 &\cdots &0 &0\\
0 &y &\ddots &\vdots &\vdots\\
\vdots &\vdots &\ddots &\ddots& \vdots\\
0 &0 &\cdots          &y &0
\end{pmatrix}.\]

It follows that $(R_{a,1}')^\flat \cong \bigoplus_{i \in \Z_a}T^i((y^2,y,y,\dots,y)) \in \MatFac{\mathbf{k}\llbracket y \rrbracket}{a}{y^{a+1}}$. By Proposition \ref{thm:finite_reduced_type} and Corollary \ref{thm:Ulrich}, the matrix factorization $(y^2,y,y,\dots,y)$ $\in \MatFac{\mathbf{k}\llbracket y \rrbracket}{a}{y^{a+1}},$ and its corresponding shifts, are the only indecomposable reduced matrix factorizations of $y^{a+1}$ with $a$ factors.

Notice that for $a \ge 4$, the polynomial $y^{a+1}$ does not appear on the list given in Theorem \ref{thm:B} for any $d > 2$. Thus, the conclusions of this example imply that, for $a\ge 4$, the polynomial $y^{a+1}$ has infinitely many isomorphism classes of indecomposable matrix factorizations with $a$ factors but only finitely many which are reduced.
\end{example}

The last example shows the necessity of the assumption $d \leq \text{ord}(f)$ in Corollary \ref{thm:Ulrich}.

\begin{example}
Let $\mathbf k$ be algebraically closed of characteristic zero. Set $S = \mathbf{k}\llbracket x \rrbracket$, $f=x^3$, and $R=S/(f)$. The hypersurface ring $R^\sharp = \mathbf{k}\llbracket x , y \rrbracket/(x^3+y^4)$ is the same ring given in Example \ref{example:E_6}, however, here we are viewing $R^\sharp$ as the $4$-fold branched cover of $R= \mathbf{k}\llbracket x \rrbracket/(x^3)$. Again using the notation of \cite[9.13]{yoshino_cohen-macaulay_1990}, we take $B = \cok \beta$ where
\[\beta = \begin{pmatrix}
y &0 &x\\
x &-y^2 &0\\
0 &x &-y
\end{pmatrix}.\]
The MCM $R^\sharp$-module $B$ is, in this case, free of rank $4$ over $S=\mathbf{k}\llbracket x \rrbracket$. In particular, if $e_1,e_2, e_3 \in B$ are the images of the standard basis on $S\llbracket y \rrbracket^3$, then an $S$-basis for $B$ is $\{e_1,e_2,e_3,ye_2\}$. Multiplication by $y$ on $B$ is therefore given by the $S$-matrix
\[\phi = \begin{pmatrix}
0 &0 &x &0\\
-x &0 &0 &0\\
0 &0 &0 &x\\
0 &1 &0 &0
\end{pmatrix}.\]
Notice that $B$ is an Ulrich $R^\sharp$-module but multiplication by $y$ on $B$ is given by a non-reduced matrix. In other words, the condition $d \leq \text{ord}(f)$ in Proposition \ref{thm:finite_reduced_type} is needed to produce reduced matrix factorizations of $f$.
\end{example}

\bibliographystyle{amsalpha}
\bibliography{Bib-BCMF}

\vspace{.1in}
\noindent \textit{Email address: }\href{mailto:gjleusch@syr.edu}{gjleusch@syr.edu}\\
\address{Department of Mathematics, Syracuse University, Syracuse NY 13244, USA}\\
\textit{URL: }\url{http://www.leuschke.org/}\\

\noindent\textit{Email address: }\href{mailto:tim.tribone@utah.edu}{tim.tribone@utah.edu}\\
\address{Department of Mathematics, University of Utah, Salt Lake City, UT 84112, USA}\\
\textit{URL: }\url{https://timtribone.com/}\\

\end{document}